\documentclass[12pt, reqno, twoside, letterpaper]{amsart}

\usepackage{paperstyle}

\usepackage{euscript}

\newcommand{\euC}{\ensuremath{\EuScript{C}}}   
\newcommand{\euG}{\ensuremath{\EuScript{G}}}   
\newcommand{\euM}{\ensuremath{\EuScript{M}}}   
\newcommand{\euR}{\ensuremath{\EuScript{R}}}   
\newcommand{\model}{{\boldsymbol\euM}}
\newcommand{\BFTmodel}{{\boldsymbol\euR}}
\newcommand{\cramermodel}{{\boldsymbol\euC}}
\newcommand{\granvillemodel}{{\boldsymbol\euG}}

\newcommand{\bx}{\boldsymbol{\tt X}}

\newcommand{\bi}{\boldsymbol{\iota}}
\newcommand{\bj}{\boldsymbol{\jmath}}

\newcommand{\PR}{\mathbb{P}}

\makeatletter
\renewcommand{\pmod}[1]{\allowbreak\mkern7mu({\operator@font mod}\,\,#1)}
\makeatother

\renewcommand{\ssum}[1]{\sum_{\substack{#1}}}  
\renewcommand{\sprod}[1]{\prod_{\substack{#1}}}  

\newcommand{\order}{\asymp}      
\renewcommand{\(}{\left(}
\renewcommand{\)}{\right)}

\newcommand{\pfrac}[2]{\left(\frac{#1}{#2}\right)}  

\begin{document}


\title[Sets satisfying Bateman-Horn]
{Sets of integers satisfying Bateman--Horn statistics}
\author[W.~Banks]{William Banks}

\address{Department of Mathematics, 
 University of Missouri, 
 Columbia MO 65211, USA.}

\email{bankswd@missouri.edu}

\author[K.~Ford]{Kevin Ford}

\address{Department of Mathematics,
 University of Illinois,
 Urbana, IL 61801, USA.}

\email{ford126@illinois.edu}

\thanks{\textbf{Keywords:}
primes, Hardy-Littlewood conjectures, Bateman-Horn conjectures, 
probabilistic models.}

\thanks{\textbf{Acknowledgements:} 
The second author was supported by National Science Foundation grant 
DMS-2301264, a Simons Collaboration grant and by a Simons Fellowship.}

\begin{abstract}
In 1962, Bateman and Horn conjectured precise asymptotics for the count
of positive integers $n\le x$ for which $f_1(n),\ldots,f_k(n)$ are all prime,
where $(f_1,\ldots,f_k)$ is an admissible $k$-tuple of polynomials
in one variable. We prove that certain random sets of integers
almost surely satisfy the Bateman--Horn asymptotics in full generality 
and with a strong error term, where we have replaced
``$f_1(n),\ldots,f_k(n)$ are all prime'' with
``$f_1(n),\ldots,f_k(n)$ all lie in the random set.''
In particular, sets of integers satisfying Bateman--Horn are plentiful.
\end{abstract}

\dedicatory{Dedicated to Roger Heath-Brown on the occasion of his 75th birthday.}

\maketitle


{\large\section{Introduction}}

\subsection{Background}\label{sec:background}
\emph{The Bateman--Horn conjecture} (BHC),
formulated by Bateman and Horn in 1962~\cite{BH62, BH65},
gives a conjectural asymptotic formula for the count of
positive integers $n \le x$ at which the values of a given admissible $k$-tuple
of polynomials are all prime.  BHC is
a far-reaching refinement of several classical conjectures,
each of which
asserts the infinitude of such $n$ in a special case but offers no
asymptotic: \emph{Dickson's prime $k$-tuples conjecture}~\cite{Dickson}
for tuples of linear polynomials, the \emph{Bunyakovsky
conjecture}~\cite{B57} for a single polynomial of arbitrary degree,
and \emph{Schinzel's Hypothesis~H}~\cite{SS58} for tuples of
arbitrary degree. The one classical predecessor that does supply an
asymptotic, the \emph{Hardy--Littlewood prime $k$-tuples
conjecture}~\cite{HL23}, is recovered by BHC in the linear case.

To state the conjecture, let $f_1, \dots, f_k \in \Z[\bx]$
be distinct irreducible polynomials with positive leading
coefficients, and for any prime $p$, let
\[
\nu_p=\nu_p(f_1,\dots,f_k)\defeq\big|\{n\bmod{p}:
f_1(n)\cdots f_k(n)\equiv 0\pmod p\}\big|.
\]
The $k$-tuple $(f_1,\dots,f_k)$ is said to be
\emph{admissible} if $\nu_p < p$
for every prime $p$; this is clearly necessary for
$f_1(n),\dots,f_k(n)$ to be simultaneously prime for 
infinitely many $n$.
The \emph{singular series} (or \emph{Bateman--Horn constant})
is defined by
\[
\fS(f_1,\dots,f_k)
\defeq\prod_p\(1-\frac{\nu_p}{p}\)\(1-\frac{1}{p}\)^{-k}.
\]
The product converges to a positive real number whenever
$(f_1,\dots,f_k)$ is admissible.

\begin{conjecture*}[Bateman--Horn \cite{BH62}]
Let $(f_1,\dots,f_k)\in\Z[\bx]^k$ be an admissible tuple
of distinct polynomials that are
irreducible over $\Q$ and have positive leading coefficients,
and let $d_j\defeq\deg(f_j)$ for each $j$. Then
\begin{equation}\label{BH}
\big|\{n\le x:f_1(n),\dots,f_k(n)\text{~are all prime}\}\big|
\sim\frac{\fS(f_1,\dots,f_k)}{d_1\cdots d_k} \, \frac{x}{(\log x)^k}
\qquad (x \to \infty).
\end{equation}
\end{conjecture*}

\newpage
\noindent For instance:
\begin{enumerate}
\item $k=1$, $f_1(\bx)=\bx$:\quad$\fS(\bx)=1$, and BHC recovers
the Prime Number Theorem.
\item $k=1$, $f_1(\bx)=a\bx+b$ with $\gcd(a,b)=1$:
\quad $\fS(a\bx+b)=\phi(a)/a$, and BHC recovers
Dirichlet's Theorem on primes in arithmetic progressions.
\item $k=2$, $f_1(\bx)=\bx$, $f_2(\bx)=\bx+2$
(twin primes): \quad BHC predicts that
\[
\pi_2(x)\sim \fS(\bx,\bx+2)\cdot \frac{x}{\log^2 x}
= 2\prod_{p\ge 3}\frac{p(p-2)}{(p-1)^2}
\cdot\frac{x}{(\log x)^2}\qquad(x\to\infty).
\]
\item $k=1$, $f_1(\bx)=\bx^2+1$ (Landau's fourth problem):
\quad $d_1=2$, and since $-1$ is a quadratic
residue modulo $p$ if and only if $p\equiv 1\pmod 4$, we have
\[
\fS(\bx^2+1)=\prod_{p}\(1-\frac{\chi_{-4}(p)}{p-1}\),
\]
where $\chi_{-4}$ is the nonprincipal character
modulo $4$; BHC predicts that
\[
\big|\{n\le x:n^2+1\text{~prime}\}\big|
\sim\frac{\fS(\bx^2+1)}{2}\cdot\frac{x}{\log x}\qquad(x\to\infty).
\]
\end{enumerate}

The only established case of BHC is $k=1$ with $f_1$ linear;
in every other case, even the infinitude of $n$ for which
$f_1(n),\ldots,f_k(n)$ are simultaneously prime remains open.
For fixed $x$, however, the formula \eqref{BH} is known to be true for ``most'' 
polynomials in various senses
(where the exceptional set depends on $x$). 
In the case of linear polynomials, see work of Tchudakoff \cite{Chudakov},
Lavrik \cite{Lavrik}, Balog \cite{Balog}, Maier and Pomerance \cite{MP},
and Kawada \cite{Kawada}.
For polynomials of arbitrary degree,  such results are far more recent; see
work of Skorobogatov and Sofos \cite{SkSo}, Browning, Sofos and Ter\"{a}v\"{a}inen
\cite{BST}, and Kravitz, Woo and Xu \cite{KWX}.

One can also ask about versions of \eqref{BH} with another
subset of $\N$ in place of the primes,
natural candidates being random sets produced to mimic the behavior of primes. 
The Cram\'er model~\cite{Cramer}, however, fails to capture
the local behavior of the primes, as we now explain.
Cram\'er's random set $\cramermodel$
is chosen by selecting each integer $n\ge 3$ to lie in $\cramermodel$
independently and with probability $1/\log n$. An easy calculation
reveals that the analog of the left side of \eqref{BH} is almost surely
asymptotic to $x/(d_1\cdots d_k\log^k x)$ for every $k$-tuple of distinct
polynomials $f_1,\ldots,f_k$, whether admissible or not. 
In particular, Cram\'er's model predicts infinitely many
simultaneous prime values even for inadmissible tuples, and even if a tuple
is admissible, the singular series $\fS(f_1,\ldots,f_k)$ is absent.
The random models of Granville~\cite{Granville} and of
Banks, Ford, and Tao~\cite{BanksFordTao} were designed to overcome 
these flaws, at least for linear shifts $f_j(\bx)=\bx+h_j$.

Our aim here is to establish an \emph{almost-sure}
version of BHC in variants of the probabilistic prime models of Granville \cite{Granville} and of
Banks, Ford, and Tao~\cite{BanksFordTao}.  We denote these random
sets as $\model_1$ and $\model_2$ and define them below.
Our theorem is stated in terms of a more accurate conjectured asymptotic for the
left side of \eqref{BH}, namely
\[
M(f_1,\ldots,f_k;x) \defeq \fS(f_1,\dots,f_k)
\int_2^x \frac{\dd u}{(\log f_1(u))\cdots (\log f_k(u))}.
\]
In the special case $k=1$, $f_1(\bx)=\bx$,
this is Gauss's approximation $\int_2^x \dd u/\log u$ for the number of primes below $x$.

\medskip
\newpage

\begin{theorem}\label{thm:main}
Fix $r\in \{1,2\}$.
Almost surely, the following holds.  For every $k \ge 1$ and every admissible tuple
$(f_1,\dots,f_k)\in\Z[\bx]^k$ of distinct polynomials that are
irreducible over $\Q$ and have positive leading coefficients,
\be\label{eq:main}
\big|\{n\le x:f_1(n),\dots,f_k(n)\in\model_r\}\big| = 
M(f_1,\ldots,f_k;x) + O\Big( x\,\er^{-(\log x)^{1/3}}\Big).
\ee
\end{theorem}

\textbf{Remark.}  The main term $M(f_1,\ldots,f_k;x)$ is $\gg x/(\log x)^k$,
hence the error term in Theorem \ref{thm:main} is, in particular, smaller by
an arbitrary power of $\log x$.  We have not attempted to optimize the error term,
nor do we address uniformity in tuples $(f_1,\dots,f_k)$.

\medskip

\subsection{The Granville random model}

A flaw in the Cram\'er model is that, almost surely, the set is uniformly distributed
modulo $p$ for every prime $p$. The set of actual primes is, by contrast, uniformly 
distributed among the residues $\{1,2,\ldots,p-1\}$ modulo~$p$ for each prime.
With $x$ fixed, Granville \cite{Granville} proposed modifying Cram\'er's model to 
capture this property, by taking  $y=(\log x)^{1-o(1)}$ and choosing
each $n$ to lie in a random set $\granvillemodel$ with probability
\[
\begin{cases}
0 &\quad\text{ if } n \text{ has a prime factor } \le y, \\
(\Theta_y\log n)^{-1} &\quad\text{ otherwise},
\end{cases}
\]
where
\[
\Theta_z \defeq \prod_{p\le z} \(1-\frac1{p} \).
\]
Mirroring the primes,
the set $\granvillemodel$ avoids the residue class $0\bmod p$
for primes $p\le y$ and, like the 
Cram\'er model, has the same global distribution as the primes. 
One thus expects that $\granvillemodel$ will satisfy the analog of 
\eqref{BH}, and this has been worked out by
T\'afula~\cite{Tafula}.

\medskip

\subsection{The Banks, Ford, Tao random model}\label{sec:BFT}
The random model of \cite{BanksFordTao} corrects the flaw in
Cram\'er's model in a different way using a \emph{random residue sieve}.
A set of linear shifts $f_j(\bx)=\bx+h_j$, $1\le j\le k$, is admissible
if $\cH\defeq\{h_1,\ldots,h_k\}$ satisfies $\nu_p=|\cH\bmod{p}|<p$ for every prime $p$.
The singular series $\fS(f_1,\dots,f_k)$ is closely approximated by
its truncation $V_\cH(z)\,\Theta_z^{-k}$
(see~\cite[Lemma~3.4]{BanksFordTao}), where
\[
V_\cH(z)\defeq\prod_{p\le z}\(1-\frac{|\cH\bmod{p}|}p\).
\]
If $z=z(x) \sim x^{1/\er^{\gamma}}$ then $\Theta_{z(x)}^{-1} \sim \log x$
and the right side of \eqref{BH} is $\sim x\, V_\cH(z(x)).$
Furthermore, $V_\cH(z)$ admits the probabilistic formulation
$V_\cH(z)=\P(\cH\subseteq\cS_z)$, where $\cS_z$ is the random sifted set
\[
\cS_z\defeq\Z \setminus \bigcup_{p\le z} (a_p\bmod{p}),
\]
with $a_p\bmod{p}$ independent and uniformly distributed over 
residue classes for each prime $p\le z$.
This shows that for admissible $\cH$, \eqref{BH} becomes
\[
\big|\{n\le x:n+h\text{~is prime for every~}h\in\cH\}\big|
\sim x\,\P(\cH\subseteq \cS_{z(x)}).
\]
In other words, \eqref{BH} asserts that the probability
that a random shift of $\cH$ consists entirely of primes is asymptotically
equal to the probability that $\cH$ lies in a randomly sifted set.
Motivated by this probabilistic interpretation, the random set
\[
\BFTmodel\defeq\{n\ge\er^2:n\in\cS_{z(n)}\},
\]
was introduced in~\cite{BanksFordTao} and has proved to be
a useful model for the primes, particularly for the analysis
of local behavior such as prime gaps.
With the specific choice of $z(x)$ in \cite{BanksFordTao},
the random $\BFTmodel$ has extremely good
Hardy--Littlewood distributions with linear shifts.  In fact, almost surely,
for $f_j(\bx) = \bx + h_j$,
\[
\big|\{n\le x:f_1(n),\ldots,f_k(n) \in \BFTmodel\}\big| = 
M(f_1,\ldots,f_k;x) + O(x^{1/2+o(1)})
\]
uniformly in a wide range of tuples $h_1,\ldots,h_k$.
However, $\BFTmodel$ fails the analog of \eqref{BH} if some $f_i$
is nonlinear or linear with coefficient larger than 1.
For example, while BHC gives an asymptotic formula for the number of primes
below $x$ of the form $4n+1$, with probability $1/2$ the random set $\BFTmodel$
contains no odd numbers, and hence no integers of the form $4n+1$.

\subsection{Our random sets}\label{sec:our-models}
Both of our random sets sieve out the residue class $0\bmod p$ for every
prime $p\le t(n)$, where
\begin{equation}\label{eq:tx defn}
t(x)\defeq \exp \Big\{ (\log x\log\log x)^{2/3}\Big\}.
\end{equation}
This shared feature ensures that the small-prime divisibility behavior
of $f(n)$ matches that of a prime for any $f\in\Z[\bx]$ with positive leading
coefficient; in particular, this repairs the flaw exhibited by
$\BFTmodel$ in the $4n+1$ example at the
end of \S\ref{sec:BFT}.

The first random set, $\model_1$, is a variant of Granville's random
set in which the threshold $y = (\log x)^{1-\eps}$ is replaced by $t(n)$.
For each $n \ge 10$, independently of all other integers, we include
$n \in \model_1$ with probability
\[
\begin{cases}
0 &\quad\text{if $n$ has a prime factor $\le t(n)$,} \\[2pt]
(\Theta_{t(n)}\log n)^{-1} &\quad\text{otherwise.}
\end{cases}
\]
Independently of our work, T\'afula \cite{Tafula} used a similar 
construction, but with $t(n)$ replaced by $(\log n)/\log\log\log n$,
and established a version of Theorem \ref{thm:main} for this random
set with a weak error term $O(x/(\log^k x \log\log x))$.

Our second random set $\model_2$ is a variant of the Banks--Ford--Tao random set
$\BFTmodel$. For each prime $p$, let $a_p\bmod{p}$ be a
uniformly chosen residue class modulo $p$,
with the choices independent from one prime to the next.
For real numbers $2\le t<z$, we denote
\[
\cS_{t,z}\defeq\left\{n\in\Z:
\begin{array}{r@{\,}l@{\quad}l}
n & \not\equiv 0\pmod{p}    & \text{for } p\le t\\
n & \not\equiv a_p\pmod{p}  & \text{for } t<p\le z
\end{array}
\right\}.
\]
The outer threshold $z(x)\defeq x^{1/\er^{\gamma}}$ 
(which is similar to the choice made in \cite{BanksFordTao}) is calibrated so that
$\Theta_{z(x)}^{-1}\sim \log x$, matching the density of primes below $x$;
with this choice, we put
\be\label{eq:BF-newmodel-defn}
\model_2\defeq\{n\ge 10 : n\in\cS_{t(n),z(n)}\}.
\ee

The two models differ structurally in one crucial respect.  In $\model_1$
the events $\{n\in\model_1\}$ are mutually independent across $n$, so
the moment calculations underlying Theorem~\ref{thm:main} reduce to
single-variable estimates.  In $\model_2$, by contrast, the residues
$a_p$ are shared across all $n$, so the events $\{n\in\model_2\}$ are
correlated; quantifying this correlation is the principal technical
task of the paper, and is carried out in~\S\ref{sec:second-model}.
The payoff is that $\model_2$ is expected to be a more faithful model
for prime-gap statistics and other local questions, such as the
prime-gap moment conjectures of Heath-Brown~\cite{HB82}.

{\large\section{Preliminaries}}

In this section we gather analytic estimates used in the proof of
Theorem~\ref{thm:main} and fix some notational conventions used in the sequel.

\bigskip
\subsection{Analytic tools}

Our first result is a standard application of Landau's Prime Ideal Theorem;
see, for example, \cite[pp. 33--38]{CM}.

\begin{lemma}\label{lem:PIT}
Let $f\in \Z[\bx]$ be irreducible over $\Q$ and non-constant.  Then
there is a constant $c_f > 0$ such that
\[
\prod_{p\le x}\(1-\frac{\rho_p(f)}{p}\)
=\frac{\er^{-\gamma}\fS(f)}{\log x}\(1+O\big(\er^{-c_f\sqrt{\log x}}\,\big)\),
\]
where
\[
\rho_p(f)\defeq\big|\{n\bmod{p}:f(n)\equiv 0\pmod p\}\big|.
\]
\end{lemma}

The special case $f(\bx)=\bx$ of Lemma~\ref{lem:PIT}, in which $\fS(\bx)=1$,
is a strong form of Mertens' third theorem,
deducible from the prime number theorem with classical error term; a stronger error
term is available using the Vinogradov--Korobov zero-free region.

We remark that, under the Generalized Riemann Hypothesis for all
Dedekind zeta functions (including the Riemann zeta function), the error term
in Lemma~\ref{lem:PIT} sharpens to $O_f(x^{-1/2}\log^2 x)$, and a straightforward
modification of the proof of Theorem~\ref{thm:main} yields the
improved error term $O\big(x\,\er^{-(\log x)^{1/2}}\big)$.

The next lemma is a consequence of the Fundamental Lemma of the sieve.

\begin{lemma}\label{lem:fundamental-lemma}
Fix a positive integer $K$, and for each prime $p$ let $I_p$ be a
set of at most $\min(p-1,K)$ residue classes modulo $p$. For any
real $v$ and any real $w \ge z \ge 10$, setting
$u \defeq \log w / \log z$, we have
\[
\big|\{v<n\le v+w:n\not\in I_p \text{ for all } p\le z\}\big|
=\Big(1+O_K\big(\er^{-(1/2)u\log u}\big)\Big)\,w\prod_{p\le z}\(1-\frac{|I_p|}{p}\).
\]
\end{lemma}

\begin{proof}
We apply \cite[Theorem~3.6]{Ford-sieve-notes} with $D=w^{2/3}$, noting that the remainders
satisfy $|r_d| \le K^{\omega(d)} \ll d^{1/100}$, where $\omega(d)$ denotes the number of
distinct prime factors of $d$, together with the lower bound
$V(z)=\prod_{p\le z}(1-|I_p|/p) \gg (\log z)^{-K}$.
One may also use \cite[Theorem~6.12]{FI}.
\end{proof} 

\subsection{Setup and notation}
Since the collection of admissible tuples $(f_1,\ldots,f_k)$
(over all $k\ge 1$) is countable, and a countable intersection of
probability-one events has probability one, it suffices to establish
\eqref{eq:main} almost surely for each tuple individually. We
therefore fix an admissible tuple $(f_1,\ldots,f_k)$ from now on,
and for brevity, we write
\[
\fS \defeq \fS(f_1,\ldots,f_k), \qquad M(x) \defeq M(f_1,\ldots,f_k;x).
\]
There is no further loss of generality in assuming that each $f_j$
has all positive coefficients and that the polynomials are ordered
so that
\begin{equation}\label{fi-ordered}
n< f_1(n) < f_2(n) < \cdots < f_k(n) \qquad (n\in\N),
\end{equation}
since this can be arranged by replacing each $f_j$ with the shift
$n \mapsto f_j(n+N)$ for sufficiently large $N$, and relabeling.
Throughout the proof, all implied constants are allowed to depend on
$f_1,\dots,f_k$ but on no other quantity, and this dependence is
suppressed in the notation.

To prove the theorem, we follow the method of \cite[\S 4]{BanksFordTao},
estimating the first and second moments of the sums
\[
\sum_{x<n\le x+y}X_n\qquad(\sqrt{x}<y\le x),
\]
where $\ind{\cdot}$ denotes the indicator function and
\[
X_n\defeq\ind{f_1(n),\dots,f_k(n)\in\model_r}.
\]
In each model we have a factorization $X_n = D_n R_n$
 into a deterministic part $D_n$ and a random part $R_n$. Explicitly,
\be\label{eq:Dn}
D_n \defeq \ind{ \text{for each } j:\;f_j(n)\not\equiv 0\pmod{p}
\text{ for all } p\le t(f_j(n))}.
\ee
For $\model_1$, the $R_n$ are independent Bernoulli random variables with
\begin{equation}\label{eq:Rn-model1}
\PR(R_n=1) = \prod_{j=1}^k \frac{\Theta_{t(f_j(n))}^{-1}}{\log f_j(n)},
\end{equation}
whereas for $\model_2$,
\be\label{eq:Rn-model2}
R_n \defeq \ind{ \text{for each } j:\;f_j(n)\not\equiv a_p\pmod{p} \text{ for all } t(f_j(n))<p\le z(f_j(n))};
\ee
in this case the $R_n$ are \emph{not} independent across $n$, because
the residues $a_p$ are shared.

The following lemma, which applies to both random sets, gives an
average bound on $D_n$ that will be used repeatedly.  For brevity,
we write
\[
\cE(x) \defeq \exp \Big\{ - (\log x)^{1/3}(\log\log x)^{1/6} \Big\}.
\]

\begin{lemma}\label{lem:Dn-average}
Let $v$ be sufficiently large and $v^{1/3} \le w\le v^{2/3}$.  Then
\[
\sum_{v<n\le v+w}D_n=\big\{1+O(\cE(v))\big\}
\frac{\fS\,w\,\er^{-\gamma k}}{\prod_{j=1}^k\log t(f_j(v))}.
\] 
\end{lemma}

\begin{proof}
Since the range of primes in the definition \eqref{eq:Dn} of $D_n$ varies with $n$,
it is convenient to use the inequalities
\be\label{eq:sandwich}
D_n' \le D_n \le D_n'',
\ee
where
\begin{alignat*}{3}
D_{n}'&\defeq\ind{\text{for each }j:
\; f_j(n)\not\equiv 0\pmod{p}\text{ for all }p\le t'_j},
&\qquad t_j'&\defeq t(f_j(v+w)),\\
D_{n}''&\defeq\ind{\text{for each }j:
\;f_j(n)\not\equiv 0\pmod{p}\text{ for all }p\le t''_j},
&\qquad t_j''&\defeq t(f_j(v)).
\end{alignat*}
We bound the averages of $D_n'$ and $D_n''$ using Lemma \ref{lem:fundamental-lemma}.
Recall that
\[
\nu_p\defeq\big|\{n\bmod{p}:
f_1(n)\cdots f_k(n)\equiv 0\pmod p\}\big|,
\]
and let us denote
\[
\cZ_p(f_j)\defeq\{n\bmod{p}:f_j(n)\equiv 0\pmod p\}
\mand
\rho_p(f_j)\defeq\big|\cZ_p(f_j)\big|
\]
for each $j$. Since $f_1, \ldots, f_k$ are distinct irreducibles over $\Q$, 
each pair $f_i, f_j$ with $i \ne j$ is coprime in $\Q[\bx]$, so the sets 
$\cZ_p(f_i)$ and $\cZ_p(f_j)$ are disjoint for all but finitely many primes $p$.
Thus, there is a constant $p_0 = p_0(f_1, \ldots, f_k)$ such that 
\be\label{eq:Z-disjoint}
\cZ_p(f_i)\cap\cZ_p(f_j)=\varnothing \qquad (p>p_0,~i\ne j).
\ee
Now let $v$ be large enough so that $t_1'>p_0$, and suppose $n$ has $D_n'=1$.
For $p \le t_1'$, all $k$ polynomials are being sieved, so $n$ must avoid 
the $\nu_p$ residues in $\bigcup_{j=1}^k \cZ_p(f_j)$.
For $t_{j-1}' < p \le t_j'$ with $2 \le j \le k$, only $f_j, \ldots, f_k$ 
remain active; by \eqref{eq:Z-disjoint}, $n$ avoids precisely
$\rho_p(f_j) + \cdots + \rho_p(f_k)$ residues modulo $p$.

In Lemma \ref{lem:fundamental-lemma}, we take $z = t_k'$. Since $w \ge v^{1/3}$
and $\log t_k'\ll (\log v \log\log v)^{2/3}$ by \eqref{eq:tx defn}, we have
\[
u = \frac{\log w}{\log t_k'} \gg \frac{(\log v)^{1/3}}{(\log\log v)^{2/3}}.
\]
In particular, $u\log u$
substantially exceeds $(\log v)^{1/3}(\log\log v)^{1/6}$, and consequently\linebreak
$e^{-(1/2) u \log u} \ll \cE(v)$.
Lemma \ref{lem:fundamental-lemma} then implies that
\[
\sum_{v<n\le v+w} D_n' = \big\{1+O(\cE(v))\big\} w \prod_{p\le t_1'} \(1-\frac{\nu_p}{p}\) 
\prod_{j=2}^k\; \prod_{t_{j-1}'<p\le t_j'}
\(1- \frac{\rho_p(f_j)+\cdots+\rho_p(f_k)}{p}\).
\]
We introduce a factor $\prod_{p\le t_1'} (1-1/p)^k(1-1/p)^{-k}$,
and complete the partial product
$\prod_{p\le t_1'}(1-\nu_p/p)(1-1/p)^{-k}$ to $\fS$.  The remaining factor is
\begin{align*}
\prod_{p>t_1'} \(1-\frac{\nu_p}{p}\)^{-1}\(1-\frac1{p}\)^k &= \(1+O\pfrac{1}{t_1'}\)
\exp\bigg\{\sum_{j=1}^k \sum_{p>t_1'} \frac{\nu_p(f_j)-1}{p} \bigg\}
=1+O(\cE(v)),
\end{align*} 
taking logarithms in Lemma \ref{lem:PIT} and using partial summation on the double sum.
We also use \eqref{eq:Z-disjoint} to express
$1-(\rho_p(f_j)+\cdots+\rho_p(f_k))/p$ as $\prod_{i\ge j}(1-\rho_p(f_i)/p)$ 
(with a  $1+O(1/t_1')$ error), and obtain
\begin{align*}
\sum_{v<n\le v+w} &D_n' =\big\{1+O(\cE(v))\big\} \fS\,w \(1+O\pfrac{1}{t_1'}\) \prod_{p\le t_1'} \(1-\frac{1}{p}\)^k\;
  \prod_{j=2}^{k} \; \prod_{t_{1}'<p\le t_j'}\(1-\frac{\rho_p(f_j)}{p}\)\\
&=\big\{1+O(\cE(v))\big\} \fS\,w  \pfrac{\er^{-\gamma}}{\log t_1'}^k \; \prod_{j=2}^{k}
  \pfrac{\log t_1'}{\log t_j'}\\
&=\big\{1+O(\cE(v))\big\} \frac{\fS\,w\,\er^{-\gamma k}}{(\log t_1')\cdots (\log t_k')}. 
\end{align*} 
An identical argument gives
\[
\sum_{v<n\le v+w} D_n'' = \big\{1+O(\cE(v))\big\}
 \frac{\fS\,w\,\er^{-\gamma k}}{(\log t_1'')\cdots (\log t_k'')}. 
\]
Since the bound $w\le v^{2/3}$ implies that
\[
\log t_j' = (\log t_j'')\big( 1 + O(w/v) \big) = \big\{1+O(\cE(v))\big\} \log t_j''
\qquad(1\le j\le k),
\]
the two sums $\sum D_n'$ and $\sum D_n''$ agree up to a factor $1+O(\cE(v))$.
Using \eqref{eq:sandwich} and recalling that $t_j''=t(f_j(v))$ for each $j$, we
complete the proof.
\end{proof}

\textbf{Remark.}  The threshold $t(x)$ was chosen to roughly balance the error terms
coming from Lemmas \ref{lem:PIT} and \ref{lem:fundamental-lemma}.

{\large\section{The random set $\model_1$}\label{sec:first-model}}

Our main task is to establish first and second moment bounds for averages of $X_n$.

\bigskip

\begin{lemma}[Moment and variance estimates]\label{lem:M1-moments}
$(i)$ With $x$ sufficiently large in terms of $f_1,\dots,f_k$,
and $\sqrt{x}<y\le x$, we have
\[
\E\bigg\{\sum_{x<n\le x+y}X_n\bigg\}
=\big\{1+O(\cE(x))\big\}M(x,y),
\]
where
\[
M(x,y)\defeq M(x+y)-M(x)=\fS
\int_x^{x+y} \frac{\dd u}{(\log f_1(u))\cdots (\log f_k(u))}.
\]
$(ii)$ For the same $x,y$,
\[
\E\bigg\{\sum_{x<n\le x+y}X_n\bigg\}^2 =\big\{1+O(\cE(x))\big\}M(x,y)^2.
\]
$(iii)$ For the same $x,y$, the variance satisfies
\[
\V\bigg\{\sum_{x<n\le x+y}X_n\bigg\} \ll \cE(x) M(x,y)^2.
\]
\end{lemma}

\begin{proof}
Partition the interval $(x,x+y]$ into $L$ subintervals
$(v_\ell, v_\ell+w_\ell]$, $1\le \ell\le L$, with
$x^{1/3} < w_\ell \le x^{2/3}$ for each $\ell$; this is possible
since $y > \sqrt{x}$.  Fix $\ell$, and let $v_\ell < n \le v_\ell + w_\ell$.
For each $j$ we have $f_j(n) = (1+O(w_\ell/v_\ell))f_j(v_\ell) =
(1+O(x^{-1/3}))f_j(v_\ell)$, and hence
\begin{align*}
\E\, R_n = \PR(R_n=1)
&= \big(1 + O(x^{-1/3}) \big) \prod_{j=1}^k
\frac{\Theta_{t(f_j(v_\ell))}^{-1}}{\log f_j(v_\ell)}\\
&= \big(1 + O(\cE(x)) \big) \prod_{j=1}^k
\frac{\er^{\gamma}\log t(f_j(v_\ell))}{\log f_j(v_\ell)},
\end{align*}
the second step following from Lemma~\ref{lem:PIT} with $f(\bx) = \bx$.
Combining this with Lemma~\ref{lem:Dn-average} yields
\begin{align*}
\E\bigg\{\sum_{v_\ell<n \le v_\ell+w_\ell}X_n\bigg\}
&= \big(1 + O(\cE(x))\big) \frac{\fS\, w_\ell}{\log f_1(v_\ell) \cdots \log f_k(v_\ell)}\\
&= \big(1 + O(\cE(x))\big)\, M(v_\ell, w_\ell).
\end{align*}
Summing over $\ell$ gives part $(i)$.

For any pair of block indices $\ell_1, \ell_2$, the independence of the
$R_n$ across $n$ gives
\[
\E\bigg\{\ssum{v_{\ell_1}<n \le v_{\ell_1}+w_{\ell_1} \\ v_{\ell_2}<m \le v_{\ell_2}+w_{\ell_2}} X_n X_m\bigg\}
= \E\bigg\{\sum_{v_{\ell_1}<n \le v_{\ell_1}+w_{\ell_1}}X_n\bigg\}\,
  \E\bigg\{\sum_{v_{\ell_2}<m \le v_{\ell_2}+w_{\ell_2}}X_m\bigg\}
+ O(w_{\ell_1})\mathbbm{1}_{\ell_1=\ell_2},
\]
where the $O(w_{\ell_1})$ term accounts for the diagonal contribution $n=m$
when $\ell_1=\ell_2$.  Since $M(v,w)\gg w/(\log v)^k$ and
$w_\ell \ge x^{1/3}$ for every $\ell$, this diagonal error is negligible, and the right side simplifies to
\[
\big(1 + O(\cE(x))\big)\, M(v_{\ell_1}, w_{\ell_1})\, M(v_{\ell_2}, w_{\ell_2}).
\]
Summing over all $\ell_1,\ell_2$ yields part $(ii)$.
Finally, part $(iii)$ follows from $(i)$ and $(ii)$ via
$\V\{S\} = \E\{S^2\} - (\E\{S\})^2$.
\end{proof}

\bigskip

\begin{proof}[Proof of Theorem~\ref{thm:main} for $r=1$]
Fix an admissible tuple $(f_1,\ldots,f_k)$, and let 
$K_0$ be a sufficiently large constant that depends only on $f_1,\ldots,f_k$.
Let $(x_m)_{m\ge 0}$ be the strictly increasing sequence defined by
\[
x_0 \defeq K_0, \qquad x_{m+1} \defeq x_m+\delta_m \quad (m\ge 0),
\qquad\text{where}\quad\delta_m\defeq x_m\,\er^{-(\log x_m)^{1/3}}.
\]
We have $x_m\to\infty$ as $m\to\infty$.  For each $m\ge 0$, let
$E_m$ be the event that
\[
\Big|\big|\{x_{m} < n\le x_{m+1}:f_1(n),\dots,f_k(n)\in\model_1\}\big|
-M(x_m,\delta_m)\Big|\ge x_m\,\er^{-10(\log x_m)^{1/3}}.
\]
Taking into account that
\be\label{eq:Mgap-bound}
M(x_m,\delta_m)\ll\frac{\delta_m}{(\log x_m)^k} \ll x_m\,\er^{-(\log x_m)^{1/3}},
\ee
Chebyshev's inequality and Lemma \ref{lem:M1-moments}\,$(iii)$ yield
\[
\P\,E_m \ll \frac{\cE(x_m)M(x_m,\delta_m)^2}{x_m^2\,\er^{-20(\log x_m)^{1/3}}}
\ll \er^{18(\log x_m)^{1/3}} \cE(x_m)
\ll \er^{-100(\log x_m)^{1/3}},
\]
where the last step uses the definition of $\cE(x_m)$
and holds for $K_0$ sufficiently large.
Since any dyadic interval $(X,2X]$ contains at most $O(\er^{(\log X)^{1/3}})$
points $x_m$, we have
\[
\sum_{X<x_m\le 2X}\P\,E_m\ll \er^{-99(\log X)^{1/3}}.
\]
Summing over $X\in\{2^j K_0:j\ge 0\}$, it follows that
$\sum_{m=0}^\infty \P\,E_m$ converges.
By the Borel--Cantelli lemma, almost surely only finitely many events $E_m$ occur.

Now suppose that $E_m$ fails for every $m\ge m_0$, and let $x>x_{m_0}$.
Define $m_1$ by $x_{m_1} < x \le x_{m_1+1}$, and split the interval
$[1,x]\subset\N$ into an initial block $[1,x_{m_0}]$,
the full blocks $(x_m,x_{m+1}]$ with $m_0\le m\le m_1-1$, and a final
block $(x_{m_1},x]$. Then
\dalign{
&\big|\{n\le x:f_1(n),\ldots,f_k(n)\in\model_1\}\big|\\
&\qquad=O(x_{m_0}+x-x_{m_1})
+\sum_{m_0\le m\le m_1-1}\Big\{M(x_m,\delta_m)
+O\big(x_m\,\er^{-10(\log x_m)^{1/3}}\big)\Big\}.
}
Since $x_{m_0}=O(1)$, $x-x_{m_1}\le\delta_{m_1}$, and
the first sum telescopes to $M(x_{m_1})-M(x_{m_0})$, 
\[
\big|\{n\le x:f_1(n),\ldots,f_k(n)\in\model_1\}\big|
=M(x_{m_1})+O\big(x_{m_1}\,\er^{-(\log x_{m_1})^{1/3}}\big).
\]
Finally, as $M(x)-M(x_{m_1})\le M(x_{m_1},\delta_{m_1})$, we may
use~\eqref{eq:Mgap-bound} to pass from $x_{m_1}$ to $x$,
and we conclude that
\[
\big|\{n\le x:f_1(n),\ldots,f_k(n)\in\model_1\}\big|
=M(x)+O\big(x\,\er^{-(\log x)^{1/3}}\big),
\]
which establishes~\eqref{eq:main}.
\end{proof}

\medskip

{\large\section{The random set $\model_2$}\label{sec:second-model}}

We establish first and second moment estimates for $R_n$,
defined in \eqref{eq:Rn-model2}. Combined with Lemma \ref{lem:Dn-average},
this gives first and second moments of sums of $X_n$.
The analysis is more involved than in Section~\ref{sec:first-model},
since for $m\ne n$ the variables $X_m$ and $X_n$ are correlated.

\medskip

\begin{lemma}\label{lem:Rn-expectation}
For sufficiently large $n$, we have
\[
\E\,R_n=\big\{1+O(\cE(n))\big\}
\,\er^{\gamma k}\prod_{j=1}^k\frac{\log t(f_j(n))}{\log f_j(n)}.
\]
\end{lemma}

\begin{proof}
Our approach is to write $\E\, R_n$ as a product over primes, grouped
according to which $f_j$ they sieve, and apply Lemma~\ref{lem:PIT}
to each piece.

We simplify the notation by  writing
\[
f_j\defeq f_j(n),\qquad
t_j\defeq t(f_j(n)),\qquad
z_j\defeq z(f_j(n)).
\]
By \eqref{fi-ordered}, we have (for sufficiently large $n$)
\[
t_1 < t_2 < \cdots < t_k < z_1 < \cdots < z_k.
\]
By the definition \eqref{eq:Rn-model2} of $R_n$, a prime $p$ participates in the sieving
condition for $f_j$ precisely when $t_j< p\le z_j$.  Splitting the range
$(t_1,z_k]$ at the points $t_2,\ldots,t_k,z_1,\ldots,z_{k-1}$ and taking
expectations, the independence of the $a_p\bmod p$ over distinct primes leads to
\[
\E\, R_n = P_1 P_2 P_3,
\]
where
\dalign{
P_1&\defeq\prod_{j=1}^{k-1}\prod_{t_j<p\le t_{j+1}}
\(1-\frac{\phi_{j,p}}{p}\),\\
P_2&\defeq\prod_{t_k<p\le z_1}
\(1-\frac{\phi_{k,p}}{p}\),\\
P_3&\defeq\prod_{j=1}^{k-1}\prod_{z_j<p\le z_{j+1}}
\(1-\frac{\phi'_{j,p}}{p}\),
}
with
\[
\phi_{j,p}\defeq\big|\{f_1,\ldots,f_j\}\bmod{p}\big|
\mand
\phi'_{j,p}\defeq\big|\{f_{j+1},\ldots,f_k\}\bmod{p}\big|.
\]
We estimate the three products individually.

Define
\[
\Delta\defeq\prod_{\bi \ne\bj}|f_{\bi} -f_{\bj}|.
\]
By \eqref{fi-ordered}, $\Delta\ne 0$, and also $\Delta=n^{O(1)}$, thus
$\Delta$ trivially has $O(\log n)$ distinct prime factors.
If $p\nmid \Delta$ then $f_i(n)\not\equiv f_j(n)\pmod{p}$ for all $i\ne j$, and thus
$\phi_{j,p}=j$.  For $p\mid\Delta$, $\phi_{j,p}\le j$.  Since $\Delta$ has $O(\log n)$
prime factors and $(1-j/p) = (1-1/p)^j(1+O(1/p^2))$, we obtain
\[
P_1 = \(1+O\pfrac{\log n}{t_1}\)\prod_{j=1}^{k-1}\,\prod_{t_j<p\le t_{j+1}}
\(1-\frac1p\)^{j}.
\]
Applying Lemma~\ref{lem:PIT} with $f(\bx)=\bx$ (Mertens' third theorem)
to estimate the inside product, and absorbing both the prefactor error
$O(\log n/t_1)$ and the error from Lemma~\ref{lem:PIT} into
$O(\cE(n))$, we get that
\[
P_1=\big\{1+O(\cE(n))\big\}\prod_{j=1}^{k-1} \pfrac{\log t_j}{\log t_{j+1}}^{j}
=\big\{1+O(\cE(n))\big\}\frac{(\log t_1)\cdots (\log t_{k-1})}{(\log t_k)^{k-1}}.
\]
Similarly, when $t_k<p\le z_1$ and $p\nmid \Delta$, $\phi_{k,p}=k$ and we deduce that
\[
P_2 = \big\{1+O(\cE(n))\big\} \pfrac{\log t_k}{\log z_1}^k.
\]
Finally, when $z_j <p \le z_{j+1}$ and $p\nmid \Delta$, $\phi'_{j,p}=k-j$,
and it follows that
\[
P_3 = \big\{1+O(\cE(n))\big\} \prod_{j=1}^{k-1} \pfrac{\log z_j}{\log z_{j+1}}^{k-j}=
\big\{1+O(\cE(n))\big\} \frac{(\log f_1)^{k-1}}{(\log f_2)\cdots (\log f_k)}.
\]
Multiplying the estimates for $P_1$, $P_2$, and $P_3$, the intermediate
logarithms telescope and we obtain
\[
P_1 P_2 P_3 = \big\{1+O(\cE(n))\big\}\prod_{j=1}^k \frac{\log t_j}{\log z_j}
= \big\{1+O(\cE(n))\big\}\,\er^{\gamma k}\prod_{j=1}^k \frac{\log t(f_j(n))}{\log f_j(n)},
\]
using $\log z_j = \log f_j(n)/\er^{\gamma}$ in the last step.
\end{proof}

\begin{lemma}[First moment estimate]\label{lem:1st}
With $x$ sufficiently large in terms of $f_1,\dots,f_k$,
and $\sqrt{x}<y\le x$, we have
\[
\E\bigg\{\sum_{x<n\le x+y}X_n\bigg\}
=\big\{1+O(\cE(x))\big\}M(x,y).
\]
\end{lemma}

\begin{proof}
Subdivide the interval $(x,x+y]$ into blocks $(v,v+w]$, where each block has
width $w=w(v)$ satisfying $v^{1/3}\le w\le v^{1/2}$;
this is possible since $y\ge x^{1/2} > 2 x^{1/3}$.
For $v<n\le v+w$ and each $j$, 
\begin{align*}
\log t(f_j(n)) &= \big( \log t(f_j(v)) \big) \big( 1 +O(w/v) \big),\\
\log f_j(n) &= \big( \log f_j(v) \big) \big( 1 +O(w/v) \big).
\end{align*}
Since $w/v \le x^{-1/2} \ll \cE(x)$,
 Lemmas \ref{lem:Rn-expectation} and \ref{lem:Dn-average} imply
\begin{align*}
\E \bigg\{ \sum_{v<n\le v+w} X_n \bigg\} &= \big\{1+O(\cE(x))\big\} \frac{\fS w}{\log f_1(v)\cdots \log f_k(v)} \\
&=\big\{1+O(\cE(x))\big\} \fS \int_v^{v+w} \frac{\dd u}{\log f_1(u)\cdots \log f_k(u)}.
\end{align*}
Summing this over all blocks $(v,v+w]$ completes the proof.
\end{proof}

The second moment estimates follow the same basic strategy,
the principal new ingredient being an estimate of $\E\, R_{n_1} R_{n_2}$ for
pairs $n_1,n_2$ that lie within a factor of two of one another.  Most such pairs
are well-behaved, but a sparse set of ``bad'' pairs must be handled
separately. Let $\cD$ denote the set of $(n_1,n_2)$ for which
$f_i(n_1)=f_j(n_2)$ for some $i,j$.  In particular, $\cD$ contains the
diagonal $\{(n,n):n\in\N\}$.  For $(n_1,n_2)\notin\cD$ the variables
$R_{n_1}$ and $R_{n_2}$ are quasi-independent, as quantified
by Lemma~\ref{lem:Rn1Rn2-expectation} below; for pairs in $\cD$ this
fails, and such pairs must be treated using trivial bounds.

\begin{lemma}\label{lem:Rn1Rn2-expectation}
For $n_1$ sufficiently large, if $n_1 < n_2 \le 2n_1$ and $(n_1,n_2)\notin \cD$, then
\[
\E\, R_{n_1} R_{n_2} = \big\{1+O(\cE(n_1))\big\}
\er^{2\gamma k} \prod_{i=1}^2 \frac{\log t(f_1(n_i))\cdots \log t(f_k(n_i))}
{\log f_1(n_i) \cdots \log f_k(n_i)}.
\]
\end{lemma}

\begin{proof}
We use the simplified notation:
\[
t_{i,j}\defeq t(f_j(n_i))\mand
z_{i,j}\defeq z(f_j(n_i)).
\]
By \eqref{fi-ordered}, we have
\[
t_{1,j}<t_{2,j}\mand z_{1,j}<z_{2,j} \qquad (1\le j\le k),
\]
and also
\begin{align*}
t_{1,1} &< \cdots < t_{1,k}<z_{1,1} < \cdots < z_{1,k}, \\
t_{2,1} &< \cdots < t_{2,k}<z_{2,1} < \cdots < z_{2,k}. 
\end{align*}
In particular, the conditions on $R_{n_1}$ and $R_{n_2}$ involve only
primes in $(t_{1,1},z_{2,k}]$.
For each relevant triple $(i,j,p)$, let
\[
R_{i,j,p} \defeq\ind{f_j(n_i)\not\equiv a_p\pmod{p}}
\]
be the indicator that the prime $p$ does not sieve the value $f_j(n_i)$.
Then
\[
R_{n_1}R_{n_2}= \prod_{i=1}^{2}\prod_{j=1}^k
\prod_{t_{i,j}<p\le z_{i,j}} R_{i,j,p}.
\]
Since $n_2\le 2n_1$, the definition \eqref{eq:tx defn}
yields the estimates
$\log t_{i,j}\asymp (\log n_1\,\log_2 n_1)^{2/3}$
and $z_{i,j}\order n_1^{d_j/\er^\gamma}$;
in particular, the $t_{i,j}$ are all of the same order of magnitude,
as are the $z_{i,j}$ for each fixed $j$.
The $R_{i,j,p}$ are independent from one prime to the next,
but they are correlated on the same
prime $p$ via the shared residue $a_p\bmod p$; taking
expectations, we get that
\[
\E\, R_{n_1}R_{n_2} = \prod_{t_{1,1}<p\le z_{2,k}}\E\,\Bigg\{
\sprod{(i,j)\in \cJ_p} R_{i,j,p} \Bigg\},
\]
where 
\[
\cJ_p\defeq\{(i,j):i\in\{1,2\},~
j\in\{1,\ldots,k\},~t_{i,j}<p\le z_{i,j}\}.
\]
Since $a_p \bmod p$ is uniform over $\mathbb{Z}/p\mathbb{Z}$,
the product $\prod_{(i,j)\in \cJ_p} R_{i,j,p}$ vanishes if and 
only if $a_p$ hits one of the residues
$\{f_j(n_i) \bmod p : (i,j) \in \cJ_p\}$; therefore,
\[
\E\,\Bigg\{
\sprod{(i,j)\in\cJ_p} R_{i,j,p}
\Bigg\}=1-\frac{\psi_p}{p},\qquad\text{where}\quad
\psi_p\defeq\big|\{f_j(n_i)\bmod{p}:(i,j)\in\cJ_p\}\big|.
\]
We have $\psi_p=|\cJ_p|$ whenever the $2k$ residues
$\{f_j(n_i) \bmod p : (i,j) \in \cJ_p\}$ are pairwise distinct
modulo $p$, which fails only when $p$ divides
\[
\Delta\defeq
\prod_{\bi \ne\bj}|f_{\bi} (n_1)-f_{\bj}(n_1)|
\cdot
\prod_{\bi \ne\bj}|f_{\bi} (n_2)-f_{\bj}(n_2)|
\cdot
\prod_{\bi ,\bj}
\big|f_{\bi} (n_1)-f_{\bj}(n_2)\big|.
\]
The first two factors are nonzero by \eqref{fi-ordered}, and the third is
nonzero for $(n_1,n_2)\not\in\cD$; as in the proof of
Lemma \ref{lem:Rn-expectation}, $|\Delta|=n_1^{O(1)}$ and so $\Delta$ has
$O(\log n_1)$ distinct prime factors.  Since $\psi_p\le 2k$ for all $p$,
the same expansion used in Lemma~\ref{lem:Rn-expectation} yields
\begin{equation}\label{second-mom-Jp}
\E\, R_{n_1}R_{n_2}= \( 1 + O\pfrac{\log n_1}{t_{1,1}}\)
\prod_{t_{1,1}<p\le z_{2,k}}\(1-\frac{1}{p}\)^{|\cJ_p|}.
\end{equation}
We write the product in \eqref{second-mom-Jp} as 
\[
\prod_{i=1}^2 \Bigg\{ \prod_{j=1}^{k-1} \prod_{t_{i,j}<p\le t_{i,j+1}} \(1-\frac{1}{p}\)^{j} \Bigg\}
\prod_{t_{i,k}<p\le z_{i,1}} \(1-\frac{1}{p}\)^k \Bigg\{ \prod_{j=1}^{k-1}  \prod_{z_{i,j}<p\le z_{i,j+1}}
\(1-\frac{1}{p}\)^{k-j} \Bigg\}.
\]
Since $\log t_{1,j}\sim\log t_{2,j}$ and $\log z_{1,j}\sim\log z_{2,j}$,
this agrees with \eqref{second-mom-Jp} up to $1+O(\cE(n_1))$.
As in the proof of Lemma \ref{lem:Rn-expectation}, using Lemma \ref{lem:PIT} 
proves the lemma.
\end{proof}

\begin{lemma}[Second moment estimate]\label{lem:2nd}
With $x$ sufficiently large in terms of $f_1,\dots,f_k$,
and $\sqrt{x}<y\le x$, we have
\[
\E\bigg\{\sum_{x<n_1, n_2\le x+y}X_{n_1}X_{n_2}\bigg\}=
\big\{1+O(\cE(x))\big\}M(x,y)^2,
\]
where $M(x,y)\defeq M(x+y)-M(x)$ as before.
\end{lemma}

\begin{proof}
As in the proof of Lemma \ref{lem:M1-moments}, we subdivide the interval
$(x,x+y]$ into $L$ blocks $(v_\ell,v_\ell+w_\ell]$
with $x^{1/3}<w_\ell \le x^{2/3}$ for each $\ell$. For each pair $(n_1,n_2)$,
if $n_i\in(v_{\ell_i},v_{\ell_i}+w_{\ell_i}]$ for $i=1,2$,
then we have for each $j$:
\dalign{
\log f_j(n_i)&=\big(1+O(x^{-1/3})\big)\log f_j(v_{\ell_i}),\\
\log t(f_j(n_i))&=\big(1+O(x^{-1/3})\big)\log t(f_j(v_{\ell_i})).
}
For a given $n_1$, there are $O(1)$ choices of $n_2$ with $(n_1,n_2)\in\cD$,
hence at most $O(w_{\ell_1})$ such pairs in the block, and each contributes
$\E\, R_{n_1}R_{n_2}\le 1$ trivially.
Combining Lemmas \ref{lem:Dn-average} and \ref{lem:Rn1Rn2-expectation}
for the remaining pairs, we find that
\[
\E\bigg\{\ssum{v_{\ell_1}<n_1\le v_{\ell_1}+w_{\ell_1}
\\v_{\ell_2}<n_2\le v_{\ell_2}+w_{\ell_2}}X_{n_1}X_{n_2}\bigg\}
=\big\{1+O(\cE(x))\big\}M(v_{\ell_1},w_{\ell_1})
M(v_{\ell_2},w_{\ell_2})+O(w_{\ell_1}).
\]
The term $O(w_{\ell_1})$ can be omitted
since $M(v,w)\gg w/(\log x)^k$ and $w_{\ell_1},w_{\ell_2}>x^{1/3}$,
and therefore
\[
M(v_{\ell_1},w_{\ell_1})M(v_{\ell_2},w_{\ell_2})
\gg \frac{w_{\ell_1} w_{\ell_2}}{(\log x)^{2k}}
\gg \frac{w_{\ell_1}x^{1/3}}{(\log x)^{2k}}\gg\frac{w_{\ell_1}}{\cE(x)}
\]
since $\cE(x)$ decays slower than any fixed power of $x$.
The proposition follows on summing over $\ell_1$ and $\ell_2$
separately and using
$\sum_{\ell=1}^L M(v_\ell,w_\ell)= M(x,y)$.
\end{proof}

Combining Lemmas \ref{lem:1st} and \ref{lem:2nd},
we obtain the desired variance bound.

\begin{corollary}\label{cor:variance}
Under the hypotheses of Lemma~\ref{lem:2nd}, we have
\[
\V\bigg\{\sum_{x<n\le x+y}X_n\bigg\}\ll \cE(x) M(x,y)^2.
\]
\end{corollary}

By an identical argument to that in the $r=1$ case,
combining Corollary~\ref{cor:variance}
with the Borel--Cantelli lemma establishes Theorem \ref{thm:main} when $r=2$.

{\large\section{Remarks on polynomials in several variables}\label{sec:multivariate}}

The Bateman--Horn conjecture (BHC), as stated in \S\ref{sec:background},
is restricted to polynomials in one variable. There are heuristic
extensions to tuples of polynomials $F_1,\ldots,F_k\in\Z[\bx_1,\ldots,\bx_m]$;
see, e.g., the appendix of~\cite{DS19}, particularly Conjecture~A.3.
A great number of special cases are known to hold, many of them classical:
counting primes of the form $x^2+y^2$, problems of Waring and Waring--Goldbach
type, and three-term arithmetic progressions of primes. There is a great
deal of current work on problems where there are enough variables for
the circle method and techniques from arithmetic geometry and the
geometry of numbers to be effective; see, for instance, the fundamental
work of Heath-Brown~\cite{HB02}. We highlight a few striking examples
where the number of variables is ``small'', starting with the theorems
of Friedlander and Iwaniec~\cite{FI98} and of Heath-Brown~\cite{HB01},
that the polynomials $x^2+y^4$ and $x^3+2y^3$, respectively, each
capture infinitely many primes. Also notable are the works of
Balog~\cite{Balog2}, Green and Tao~\cite{GT1,GT2}, and Green, Tao
and Ziegler~\cite{GTZ} on prime solutions of systems of linear equations.

The methods of the present paper extend to the multivariate setting,
yielding an analog of~\eqref{eq:main} (as formulated in~\cite{DS19})
that holds almost surely for every admissible tuple of multivariate
polynomials $(F_1,\ldots,F_k)$. There are additional technicalities;
for example, one must restrict $(\bx_1,\ldots,\bx_m)$ to expanding
regions on which the leading form of each $F_i$ is positive. But the
arguments above are robust enough to handle this. In particular, the
prime number theorem for schemes over $\Z$ 
replaces Lemma~\ref{lem:PIT}, and the disjointness~\eqref{eq:Z-disjoint}
holds for $F_i,F_j$ that are coprime in $\Q[\bx_1,\ldots,\bx_m]$.


\begin{thebibliography}{99}

\bibitem{Balog} \textsc{Antal Balog},
\emph{The prime $k$-tuplets conjecture on average},
 Analytic number theory (Allerton Park, IL, 1989.  Bruce C. Berndt, Harold G. Diamond, Heini Halberstam and Adolf Hildebrand, eds.), 47--75, Progress in Mathematics, 85, Birkh\"auser Boston, Boston, MA, 1990.

\bibitem{Balog2}\textsc{Antal Balog},
\emph{Linear equations in primes.}
Mathematika \textbf{39} (1992), no. 2, 367--378.

\bibitem{B57}
\textsc{V.~Bouniakowsky},
\emph{Sur les diviseurs num\'eriques invariables des fonctions rationnelles enti\`eres},
M\'em.\ Acad.\ Sci.\ St.\ P\'etersbourg, S\'er.~6, \textbf{6} (1857), 305--329.

\bibitem{BanksFordTao}
\textsc{William Banks, Kevin Ford, and Terence Tao},
\emph{Large prime gaps and probabilistic models},
Inventiones\ Mathematicae\ \textbf{233} (2023), no.~3, 1471--1518.

\bibitem{BH62}
\textsc{Paul~T.~Bateman and Roger~A.~Horn},
\emph{A heuristic asymptotic formula concerning the distribution of prime numbers},
Mathematics of Computation\ \textbf{16} (1962), 363--367.

\bibitem{BH65}
\textsc{Paul~T.~Bateman and Roger~A.~Horn},
\emph{Primes represented by irreducible polynomials in one variable},
Proceedings of the Symposium in Pure Mathematics, Vol.~VIII, American\ Mathematical\ Society, Providence, RI, 1965, pp.~119--132.

\bibitem{BST}
\textsc{Tim Browning, Efthymios Sofos and Joni Ter\"{a}v\"{a}inen},
\emph{Bateman-Horn, polynomial Chowla and the Hasse principle with probability 1},
preprint.  \url{arxiv.org/2212.10373}

\bibitem{CM}
\textsc{Alina~C.~Cojocaru and M.~Ram~Murty},
\emph{An Introduction to Sieve Methods and their Applications},
London Mathematical\ Society\ Student Texts, vol.~66, Cambridge University\ Press, Cambridge, 2006.

\bibitem{Cramer}
\textsc{Harald Cram\'er},
\emph{Some theorems concerning prime numbers},
Acta Arithmetica \textbf{2} (1936), no.~1, 23--46.

\bibitem{Dickson}
\textsc{L.~E.~Dickson},
\emph{A new extension of Dirichlet's theorem on prime numbers},
Messenger of Mathematics\ \textbf{33} (1904), 155--161.

\bibitem{DS19}
\textsc{Kevin Destagnol and Efthymios Sofos},
\emph{Rational points and prime values of polynomials in moderately many variables},
Bulletin des Sciences Math\'ematiques\ \textbf{156} (2019), 102794.

\bibitem{Ford-sieve-notes}
\textsc{Kevin Ford},
\emph{Sieve methods lecture notes},
available at \url{https://ford126.web.illinois.edu/sieve2023.pdf}.

\bibitem{FI98}
\textsc{John~Friedlander and Henryk~Iwaniec},
\emph{The polynomial $X^2+Y^4$ captures its primes},
Annals of Mathematics (2)\ \textbf{148} (1998), no.~3, 945--1040.

\bibitem{FI}
\textsc{John~Friedlander and Henryk~Iwaniec},
\emph{Opera de Cribro},
American\ Mathematical\ Society\ Colloq.\ Publ., vol.~57, 
 Providence, RI, 2010.

\bibitem{Granville}
\textsc{Andrew Granville}, 
\emph{Harald Cram\'er and the distribution of prime numbers},
Harald Cram\'er Symposium (Stockholm, 1993).
\emph{Scandinavian Actuarial\ Journal} (1995), no.~1, 12--28.

\bibitem{GT1} \textsc{Ben~J.~Green and Terence~C.~Tao}, \emph{Linear equations in primes,} 
Annals of Mathematics \textbf{171} (2010), no. 3, 1753--1850.

\bibitem{GT2} \textsc{Ben~J.~Green and Terence~C.~Tao}, \emph{The quantitative behaviour of polynomial orbits on nilmanifolds}, Annals of Mathematics \textbf{175} (2012), no. 2, 465--540.

\bibitem{GTZ}  \textsc{Ben~J.~Green, Terence~C.~Tao and Tamar~Ziegler}, 
\emph{An inverse theorem for the Gowers $U^{s+1}[N]$-norm},
Annals of Mathematics \textbf{176} (2012), 1231--1372.

\bibitem{HL23}
\textsc{G.~H.~Hardy and J.~E.~Littlewood},
\emph{Some problems of ``Partitio numerorum''; III: On the expression of a number as a sum of primes},
Acta Mathematica\ \textbf{44} (1923), 1--70.

\bibitem{HB82}
\textsc{D.~R.~Heath-Brown},
\emph{Gaps between primes, and the pair correlation of zeros of the zeta-function},
Acta Arithmetica\ \textbf{41} (1982), no.~1, 85--99.

\bibitem{HB01}
\textsc{D.~R.~Heath-Brown},
\emph{Primes represented by $x^3+2y^3$},
Acta Mathematica\ \textbf{186} (2001), no.~1, 1--84.

\bibitem{HB02}
\textsc{D.~R.~Heath-Brown},
\emph{The density of rational points on curves and surfaces.}
Annals of Mathematics (2) {\bf 155} (2002), no. 2, 553--595.

\bibitem{Kawada}
\textsc{K.~Kawada}, \emph{A Montgomery-Hooley type theorem for prime $k$-tuplets},
Acta Mathematica Hungarica \textbf{66} (1995), no. 3, 177--200.

\bibitem{KWX}
\textsc{Noah Kravitz, Katherine Woo and Max Wenqiang Xu},
\emph{The distribution of prime values of random polynomials},
preprint.  \url{arxiv.org/2512.03292}


\bibitem{Lavrik}
\textsc{A. F. Lavrik}, \emph{On the theory of distribution of primes based on I. M.
Vinogradov's method of trigonometric sums}, Trudy Mat. Steklov \textbf{64} (1961), 90--125.
(Russian)

\bibitem{MP}
\textsc{Helmut Maier and Carl Pomerance},
\emph{Unusually large gaps between primes}, 
Transactions of the American Mathematical Society
\textbf{322} (1990), No. 1, 210--237.


\bibitem{SkSo}
\textsc{Alexei Skorobogatov and Efthymios Sofos},
\emph{Schinzel hypothesis on average and rational points},
Inventiones Mathematicae \textbf{231} (2023), no. 2, 673--739.

\bibitem{SS58}
\textsc{A.~Schinzel and W.~Sierpi\'nski},
\emph{Sur certaines hypoth\`eses concernant les nombres premiers},
Acta Arithmetica\ \textbf{4} (1958), 185--208.

\bibitem{Tafula}
\textsc{Christian T\'afula},
\emph{A note on the Cram\'er--Granville model},
Archiv der Mathematik \textbf{126} (2026), 275--283.

\bibitem{Chudakov}
\textsc{N.~G.~Tchudakoff}, \emph{On Goldbach--Vinogradov's theorem}, Annals of Mathematics
\textbf{48} (1947), 515--545.


\end{thebibliography}
\end{document}